\documentclass[a4paper,12pt]{article}
\usepackage[T1]{fontenc}
\usepackage{lmodern,amsmath,amsthm,amsfonts,amssymb,graphicx,float,microtype,thmtools,underscore,mathtools,xurl}
\usepackage[dvipsnames,svgnames,table]{xcolor}
\usepackage[shortlabels]{enumitem}
\setlist[itemize]{topsep=0ex,itemsep=0ex,parsep=0ex}
\setlist[enumerate]{topsep=0ex,itemsep=0ex,parsep=0ex}
\usepackage{pstricks,multido,pstricks-add,pst-node}
\usepackage[unicode=true]{hyperref}
\hypersetup{%
colorlinks,
breaklinks=true,
linkcolor={blue!60!black},
citecolor={black},
urlcolor={blue!60!black},
pdftitle={Expansion of Gap-Planar Graphs}}
\usepackage[capitalise, compress, nameinlink, noabbrev]{cleveref}
\crefname{lem}{Lemma}{Lemmas}
\crefname{thm}{Theorem}{Theorems}
\crefname{cor}{Corollary}{Corollaries}
\newcommand{\defn}[1]{\textcolor{Maroon}{\emph{#1}}}

\usepackage[longnamesfirst,numbers,sort&compress]{natbib}
\makeatletter
\def\NAT@spacechar{~}
\makeatother
\usepackage[tmargin=33mm,bmargin=33mm,lmargin=30mm,rmargin=30mm]{geometry}
\renewcommand{\baselinestretch}{1.1}
\allowdisplaybreaks

\DeclarePairedDelimiter{\ceil}{\lceil}{\rceil}

\renewcommand{\epsilon}{\varepsilon}
\renewcommand{\emptyset}{\varnothing}

\renewcommand{\geq}{\geqslant}
\renewcommand{\leq}{\leqslant}
\DeclareMathOperator{\polylog}{polylog}
\DeclareMathOperator{\dist}{dist}
\DeclareMathOperator{\sreach}{reach}
\DeclareMathOperator{\scol}{col}

\DeclareMathOperator{\eg}{eg}
\DeclareMathOperator{\ltw}{ltw}
\DeclareMathOperator{\rtw}{rtw}
\DeclareMathOperator{\tw}{tw}

\newcommand{\minorof}{\preccurlyeq}

\newcommand{\GG}{\mathcal{G}}

\newcommand{\TT}{\mathcal{T}}

\renewcommand{\thefootnote}{\fnsymbol{footnote}}
\theoremstyle{plain}
\newtheorem{thm}{Theorem}
\newtheorem{lem}[thm]{Lemma}
\newtheorem{cor}[thm]{Corollary}

\newtheorem{prop}[thm]{Proposition}

\crefname{obs}{Observation}{Observations}
\newtheorem*{lem*}{Lemma}
\theoremstyle{definition}

\newtheorem*{conj*}{Conjecture}

\begin{document}
\title{\bf\boldmath\fontsize{18pt}{18pt}\selectfont
Expansion of Gap-Planar Graphs}

\author{David~R.~Wood\,\footnotemark[2]}


\maketitle

\begin{abstract}
A graph is $k$-gap-planar if it has a drawing in the plane such that every crossing can be charged to one of the two edges involved so that at most $k$ crossings are charged to each edge. We show this class of graphs has linear expansion. In particular, every $r$-shallow minor of a $k$-gap-planar graph has density $O(rk)$. Several extensions of this result are proved: for topological minors,  for $k$-cover-planar graphs, for $k$-gap-cover-planar graphs, and for drawings on any surface. Application to graph colouring are presented.
\end{abstract}

\footnotetext[2]{School of Mathematics, Monash University, Melbourne, Australia (\texttt{david.wood@monash.edu}). Research supported by the Australian Research Council and by NSERC. }

\renewcommand{\thefootnote}{\arabic{footnote}}

\section{Introduction}

Beyond planar graphs is a field of research that studies drawings\footnote{A \defn{drawing} of a graph $G$ represents each vertex of $G$ by a distinct point in the plane, and represents each edge $vw$ of $G$ by a non-self-intersecting curve between $v$ and $w$, such that no three edges cross at a single point.}  of graphs\footnote{We consider simple, finite, undirected graphs $G$ with vertex-set $V(G)$ and edge-set $E(G)$. A graph $H$ is a \defn{minor} of a graph $G$, written \defn{$H\minorof G$}, if a graph isomorphic to $H$ can be obtained from $G$ by vertex deletions, edge deletions, and edge contractions. 
If $H$ is not a minor of $G$, then $G$ is \defn{$H$-minor-free}. 
A \defn{class} is a collection of graphs, closed under isomorphism. 
A class $\GG$ of graphs is \defn{monotone} if for every graph $G\in\GG$, every subgraph of $G$ is in $\GG$. 
A class $\GG$ of graphs is \defn{hereditary} if for every graph $G\in\GG$, every induced subgraph of $G$ is in $\GG$. 
A class $\GG$ of graphs is \defn{minor-closed} if for every graph $G\in\GG$, every minor of $G$ is in $\GG$.}
 in which the crossings are controlled in some way; see \citep{DLM19,Hong20} for surveys. This paper studies the graph-theoretic structure of certain beyond planar graph classes. As a warm-up, consider the following classical example. For an integer $k\geq 0$, a drawing of a graph $G$ is \defn{$k$-planar} if each edge is in at most $k$ crossings; see \cref{K6} for an example. A graph is \defn{$k$-planar} if it has a $k$-planar drawing. Such graphs share many structural properties of planar graphs. For example, $k$-planar graphs on $n$ vertices have treewidth\footnote{For a non-empty tree $T$, a \defn{$T$-decomposition} of a graph $G$ is a collection $(B_x:x \in V(T))$ such that (a) $B_x\subseteq V(G)$ for each $x\in V(T)$, (b) for each edge ${vw \in E(G)}$, there exists a node ${x \in V(T)}$ with ${v,w \in B_x}$, and (c) for each vertex ${v \in V(G)}$, the set $\{ x \in V(T) : v \in B_x \}$ induces a non-empty (connected) subtree of $T$. The \defn{width} of such a $T$-decomposition is ${\max\{ |B_x| : x \in V(T) \}-1}$. A \defn{tree-decomposition} is a $T$-decomposition for any tree $T$. The \defn{treewidth} of a graph $G$, denoted \defn{$\tw(G)$}, is the minimum width of a tree-decomposition of $G$. Treewidth is the standard measure of how similar a graph is to a tree; see \citep{Reed97,Bodlaender98,HW17} for surveys on treewidth. Note that if $\tw(G)\leq k$ then $G$ is $K_{k+2}$-minor-free.} $O((kn)^{1/2})$ \citep{GB07,DEW17}, generalising the Lipton-Tarjan balanced separator theorem, which essentially says that $n$-vertex planar graphs have treewidth $O(n^{1/2})$. More generally, for fixed $k\geq 0$, the class of $k$-planar graphs has bounded layered treewidth~\citep{DEW17} and bounded row treewidth~\citep{DMW17,HW24}\footnote{A \defn{layering} of a graph $G$ is a partition $(V_1,\dots,V_n)$ of $V(G)$ such that for each edge $vw\in E(G)$, if $v\in V_i$ and $w\in V_j$ then $|i-j|\leq 1$. The \defn{layered treewidth} of a graph $G$. denoted \defn{$\ltw(G)$}, is the minimum integer $k$ such that $G$ has a tree-decomposition $(B_x:x\in V(T))$ and a layering $(V_1,\dots,V_n)$ where $|B_x\cap V_i|\leq k$ for each $x\in V(T)$ and $i\in\{1,\dots,n\}$. Every planar graph has layered treewidth at most $3$ \cite{DMW13}, and every $k$-planar graph has layered treewidth at most $6(k+1)$ \citep{DEW17}. The \defn{row treewidth} of a graph $G$, denoted \defn{$\rtw(G)$}, is the minimum treewidth of a graph $H$ such that $G$ is isomorphic to a subgraph of $H\boxtimes P$, where $\boxtimes$ is the strong product and $P$ is a path. It is easily seen that $\ltw(G)\leq\rtw(G)+1$. Every planar graph has row treewidth at most $6$ \cite{DJMMUW20,UWY22}, and every $k$-planar graph has row treewidth at most $O(k^5)$ \citep{DMW23}. Note that if $G$ has radius at most $r$, then $\tw(G)\leq (2r+1)\ltw(G)-1$. }.

\begin{figure}[!h]
    \centering
    \includegraphics[width=0.4\linewidth]{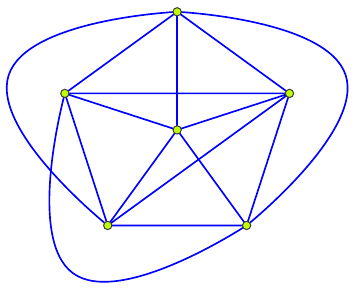}
    \caption{1-planar drawing of $K_6$}
    \label{K6}
\end{figure}

This paper focuses on the following more general beyond planar graph class. For an integer $k\geq 0$, a drawing of a graph $G$ is \defn{$k$-gap-planar} if every crossing can be charged to one of the two edges involved so that at most $k$ crossings are charged to each edge. A graph is \defn{$k$-gap-planar} if it has a $k$-gap-planar drawing.  This definition is due to \citet{GapPlanar18}. Similar definitions were independently introduced by various other authors around the same time\footnote{The \defn{crossing graph} of a drawing of a graph $G$ is the graph with vertex set $E(G)$, where two vertices are adjacent if the corresponding edges in $G$ cross. A graph $G$ is $k$-gap-planar if and only if  $G$  has a drawing in the plane such that the crossing graph has an orientation with outdegree at most $k$ at every vertex. \citet{Hakimi65} proved that any graph $H$ has such an orientation if and only if every subgraph of $H$ has average degree at most $2k$. So a graph $G$ is $k$-gap-planar if and only if $G$ has a drawing such that every subgraph of the crossing graph has average degree at most $2k$ if and only if $G$ has a drawing such that every subgraph $G'$ of $G$ has at most $k\,|E(G')|$ crossings in the induced drawing of $G'$. 
\citet{OOW19} defined a graph $G$ to be \defn{$k$-close to Euler genus $g$} if every subgraph $G'$ of $G$ has a drawing in a surface of Euler genus at most $g$ with at most $k\,|E(G')|$ crossings. The only difference between ``$k$-close to planar'' and ``$k$-gap-planar'' is that a $k$-gap-planar graph has a single drawing in which every subgraph has the desired number of crossings. A graph $H$ is \defn{$d$-degenerate} if every subgraph of $H$ has minimum degree at most $d$, which is equivalent to $H$ having an acyclic edge-orientation such that each vertex has indegree at most $d$. \citet{EG17} defined a graph to be a \defn{$d$-degenerate crossing graph} if it admits a drawing whose crossing graph is $d$-degenerate. Equivalently, $G$ has a drawing in which the crossing graph has an acyclic orientation with outdegree at most $k$ at every vertex. Thus every $k$-degenerate crossing graph is $k$-gap-planar graph, and every $k$-gap-planar graph is a $2k$-degenerate crossing graph.}. 

Various papers~\citep{EG17,HW24,HW22,HIMW24,OOW19} have studied the structure of $k$-gap-planar graphs. The main positive result, which can be concluded from the work of \citet{EG17}, says that $k$-gap-planar $n$-vertex graphs have treewidth $O(k^{3/4}n^{1/2})$ (see \cref{Treewidth}). On the other hand, recent results show that $k$-gap-planar graphs do not share analogous properties to planar graphs or $k$-planar graphs~\citep{HW24,HIMW24}. For example, \citet{HIMW24} constructed $1$-gap-planar graphs with radius $1$ and arbitrarily large complete graph minors, and thus with unbounded treewidth, unbounded layered treewidth, and unbounded row treewidth. This is in sharp contrast to the case of $k$-planar graphs, and shows that $k$-gap-planar graphs are more general than $k$-planar graphs, in the sense that for any $k$ there is a 1-gap-planar graph that is not $k$-planar. These negative results motivate further study of the class of $k$-gap-planar graphs. 

The primary contribution of this paper is to show a new positive result about the structure of $k$-gap-planar graphs. In particular, we show they have linear expansion, which is a strong property in the Graph Sparsity Theory of \citet{Sparsity}. 

To explain this result, several definitions are needed. A \defn{model} of a graph $H$ in a graph $G$ is a function $\mu$ such that:
\begin{itemize}
    \item for each $v \in V(H)$, $\mu(v)$ is a non-empty connected subgraph of $G$,
    \item $\mu(v) \cap \mu(w) = \emptyset$ for all distinct $v, w \in V(H)$, and
    \item for every edge $vw \in E(H)$, there is an edge of $G$ between $\mu(v)$ and $\mu(w)$. 
\end{itemize}
Clearly, $H$ is a minor of $G$ if and only if there is a model of $H$ in $G$. For an integer $r\geq 0$, a model $\mu$ of a graph $H$ in a graph $G$ is \defn{$r$-shallow} if $\mu(v)$ has radius at most $r$, for each $v\in V(H)$. A graph $H$ is an $r$-shallow minor of a graph $G$, written \defn{$H\minorof_r G$}, if there is an $r$-shallow model of $H$ in $G$. The \defn{density} of a graph $G$ is $|E(G)|/|V(G)|$ if $V(G)\neq\emptyset$, and is 0 if $V(G)=\emptyset$. For a graph $G$, let \defn{$\nabla_r(G)$} be the maximum density of an $r$-shallow minor of $G$. A graph class $\mathcal{G}$ has \defn{bounded expansion} if there is a function $f$ such that $\nabla_r(G)\leq f(r)$ for every graph $G\in\mathcal{G}$ and integer $r\geq 0$. Often the magnitude of such a function $f$ matters. A graph class $\mathcal{G}$ has \defn{polynomial expansion} if there exists  $c\in\mathbb{R}$ such that $\nabla_r(G)\leq c(r+1)^c$ for every graph $G\in\mathcal{G}$ and integer $r\geq 0$. As an illustrative example, the class of graphs with maximum degree at most 3 has bounded expansion (with expansion function $f(r)\in O(2^r)$), but does not have polynomial expansion. A graph class $\mathcal{G}$ has \defn{linear expansion} if there exists  $c\in\mathbb{R}$ such that $\nabla_r(G)\leq c(r+1)$ for every graph $G\in\mathcal{G}$ and integer $r\geq 0$. As another illustrative example, the class of 3-dimensional grid graphs has polynomial expansion (with expansion function $f(r)\in O(r^2)$), but does not have linear expansion. 

What can be said about the expansion of $k$-gap-planar graphs? It can be concluded from the above-mentioned treewidth bound that for each integer $k\geq 0$, the class of $k$-gap-planar graphs has polynomial expansion (see \cref{kGapPlanarNabla}). The first  contribution of this paper improves this result to linear. 

\begin{thm}
\label{LinearExpansion}
For every $k$-gap-planar graph $G$ and integer $r\geq 0$,
 $$\nabla_r(G)\leq 18(k+1)(r+1).$$
\end{thm}

We obtain a similar result for topological minors. For an integer $r\geq 0$, a graph $H$ is an \defn{$r$-shallow topological minor} of a graph $G$ if a $(\leq 2r)$-subdivision of $H$ is a subgraph of $G$. For a graph $G$ and integer $r\geq 0$, let \defn{$\widetilde{\nabla}_r(G)$} be the maximum density of an $r$-shallow topological minor of $G$. 
Every $r$-shallow topological minor is an $r$-shallow minor, so $\widetilde{\nabla}_r(G)\leq\nabla_r(G)$.  Thus, the bound in \cref{LinearExpansion} also applies to  $\widetilde{\nabla}_r(G)$. 
We in fact establish the following improved bound. 

\begin{thm}
\label{LinearTopoExpansion}
For every $k$-gap-planar graph $G$ and integer $r\geq 0$,
 $$\widetilde{\nabla}_r(G)\leq 8\sqrt{(2r+1)(k+1)}.$$
\end{thm}

\cref{LinearExpansion} is proved in \cref{LinearExpansionSection}.
A key feature of this proof is that it holds for a more general class, so-called $k$-gap-cover-planar graphs, which is of independent interest. \cref{LinearTopoExpansion} is proved in \cref{TopoExpansionSection}. \cref{Surfaces} considers generalisations of all these results for graphs drawn on arbitrary surfaces. \cref{ColouringNumbers} present applications of the above results to graph colouring. Throughout the paper, we use explicit non-optimised constants. 

\section{Background}
\label{Background}

We start with some useful lemmas from the literature, whose proofs we include for completeness. \citet{PachToth97} proved the following result for $k$-planar graphs. The proof works for $k$-gap-planar graphs (which is implicit in \citep{OOW19}). 

\begin{lem}
\label{Extremal}
Every $k$-gap-planar graph has density at most $8\sqrt{k+1}$.
\end{lem}

\begin{proof}
Say there are $c$ crossings in a $k$-gap-planar drawing of a graph $G$. Let $n:=|V(G)|$ and $m:=|E(G)|$. So $c\leq km$. The $k=0$ case follows from Euler's Formula for planar graphs, so we may assume that $k\geq 1$ and $m\geq 4n$. Thus, $c \geq \frac{m^3}{64n^2}$ by the Crossing Lemma~\citep{Ajtai82,Leighton83,Proofs1}. Hence, $km \geq c \geq \frac{m^3}{64n^2}$ and $m\leq 8\sqrt{k}n$. 
\end{proof}

The constant in \cref{Extremal} can be tweaked by applying results that improve the constant in the Crossing Lemma~\citep{PRTT-DCG06}.

The next result, which is implicit in the work of \citet{EG17}, bounds the treewidth of $k$-gap-planar graphs.

\begin{lem}
\label{Treewidth}
For every $k$-gap-planar graph $G$ on $n$ vertices, 
$$\tw(G)\leq 21(k+1)^{3/4}n^{1/2}.$$
\end{lem}

\begin{proof}
Fix a $k$-gap-planar drawing of $G$ with $c$ crossings. Let $m:=|E(G)|$. Fix an arbitrary edge-orientation of $G$. Let $G'$ be the planar graph obtained by adding a dummy vertex at each crossing point. Say $G'$ has $n'$ vertices. Thus $n'=n+c\leq n+ km \leq n + 8(k+1)^{3/2}n \leq 9(k+1)^{3/2}n$ by \cref{Extremal}. Every planar graph on $n'$ vertices has treewidth at most $2\sqrt{3n'}$ \citep{DMW17}. So $\tw(G')\leq 2\sqrt{3n'}$. Consider a tree-decomposition $\TT$ of $G'$ with width at most $2\sqrt{3n'}$. For each dummy vertex $z\in V(G')$, if $z$ is at the crossing point of edges $\overrightarrow{vw}$ and $\overrightarrow{ab}$ in $G$, then replace each instance of $z$ in a bag of $\TT$ by $v$ and $a$. We obtain a tree-decomposition of $G$ with width at most $4\sqrt{3n'} < 21(k+1)^{3/4} n^{1/2}$.
\end{proof}


\citet{DN16} (also see \citep{Dvorak21,Dvorak16,Dvorak18}) showed that a hereditary graph class $\GG$ has polynomial expansion if and only if every graph in $\GG$ admits strongly sublinear separators, which is equivalent to saying that every $n$-vertex graph in $\GG$ has treewidth $O(n^{1-\epsilon})$ for some fixed $\epsilon>0$. \citet{EG17} used this result to conclude that $k$-gap-planar graphs have polynomial expansion. To obtain the best known result using this approach we apply the following result of \citet{ER18}.

\begin{thm}[{\citep[Theorem~4]{ER18}}]
\label{ER}
For any $\alpha > 0$ and $\delta\in(0,1]$, if a monotone class $\GG$ has the property that every $n$-vertex
graph in $\GG$ has a balanced separator of order at most $\alpha n^{1-\delta}$, then $\GG$ has expansion bounded by the function 
$$f:r\to \big(c_1\, \alpha \log(\alpha+1)\cdot (r+1)\big)^{1/\delta} \big( \tfrac{1}{\delta} \log(r+3) \big)^{c_2/\delta},$$ 
where $c_1$ and $c_2$ are absolute constants. 
\end{thm}

The balanced separator property in \cref{ER} is implied if every $n$-vertex graph in $\GG$ has treewidth at most $\alpha n^{1-\delta}$ (by \citep[(2.6)]{RS-II}). Thus, for $k$-gap-planar graphs,  \cref{ER} is applicable with $\alpha=21(k+1)^{3/4}$ and $\delta=\frac12$ by \cref{Treewidth}. Hence:

\begin{prop}
\label{kGapPlanarNabla}
For every $k$-gap-planar graph $G$, 
$$\nabla_r(G) \leq 
c_1\,k^{3/2} \log^2(k +1) \cdot (r+1)^{2} ( \log(r+3) )^{c_2},$$
where  $c_1$ and $c_2$ are absolute constants.
\end{prop}

\cref{kGapPlanarNabla} says that $k$-gap-planar graphs have  polynomial $O(r^2\polylog r)$ expansion. The next sections give improved bounds.  


\section{Shallow Minors}
\label{LinearExpansionSection}

This section proves that $k$-gap-planar graphs have linear expansion (\cref{LinearExpansion}). A key idea in the proof is to consider the following more general graph class that is implicit in the work of \citet{AFPS14}, and explicitly introduced by \citet{HKW}.
For a graph $G$ and set $S\subseteq E(G)$, a \defn{vertex-cover} of $S$ is a set $C\subseteq V(G)$ such that every edge in $S$ has at least one endpoint in $C$. 

Distinct edges $e$ and $f$ in a graph are \defn{independent} if $e$ and $f$ have no common endpoint. Consider a drawing $D$ of a graph $G$. Let $D^\times$ be the set of pairs of independent edges that cross in $D$. A \defn{$D$-cover} of an edge $e\in E(G)$ is a vertex-cover of the set of edges $f$ such that $\{e,f\}\in D^\times$. Any minimal $D$-cover of $e$ does not include an endpoint of $e$. So it suffices to consider $D$-covers of an edge $vw$ in $G-v-w$. For an integer $k \geq 0$, a drawing $D$ of a graph $G$ is \defn{$k$-cover-planar} if each edge $vw\in E(G)$ has a $D$-cover in $G-v-w$ of size at most $k$. A graph $G$ is \defn{$k$-cover-planar} if $G$ has a $k$-cover-planar drawing\footnote{There is a slight difference between this definition of $k$-cover-planar and the definition of $k$-cover-planar given by \citet{HKW}, who consider vertex-covers of the set of all edges that cross an edge $vw$ (including edges that cross $vw$ and are incident to $v$ or $w$). Every $D$-cover under this definition is a $D$-cover under our definition, and every $k$-cover-planar graph under our definition is $(k+2)$-cover-planar under the definition of \citet{HKW}, by simply adding $v$ and $w$ to the cover of each edge $vw$. We work with this definition since it matches the approach of \citet{AFPS14}, and is the strongest setting where the proof of \cref{ExtremalGapCover} works.}.

\citet{HKW} showed that if a graph $G$ has a $k$-cover-planar drawing with no $t$ pairwise crossing edges incident to a single vertex, then $G$ has row treewidth at most some function $f(k,t)$. This implies that $G$ has layered treewidth at most $f(k,t)$, and the class of such graphs has linear expansion by a result of \citet{DMW17}. However, the dependence on $k$ and $t$ is exponential. We give much improved bounds on the expansion with no dependence on $t$. Moreover, our results hold in the following more general setting.

The definitions of $k$-gap-planar and $k$-cover-planar are merged as follows. Consider a drawing $D$ of a graph $G$. Recall that $D^\times$ is the set of pairs of independent edges that cross in $D$. A \defn{bearing} of $D$ is a set $B$ of ordered pairs $(e,f)$ with $\{e,f\}\in D^\times$, such that for each $\{e,f\}\in D^\times$ at least one of $(e,f)$ and $(f,e)$ (possibly both) is in $B$. For a bearing $B$ of $D$, a \defn{$B$-cover} of an edge $e\in E(G)$ is a vertex-cover of the set of all edges $f$ such that $(e,f)\in B$. For an integer $k \geq 0$, a drawing $D$ of a graph $G$ is \defn{$k$-gap-cover-planar} if there is a bearing $B$ of $D$ such that each edge has a $B$-cover in $G-v-w$ of size at most $k$.

Every $k$-gap-planar drawing is $k$-gap-cover-planar. But $k$-gap-planar drawings are strictly more general. To see this, 
as illustrated in \cref{GapCoverPlanar}, consider the graph drawing $D$ consisting of $n$ disjoint plane stars $S_1,\dots,S_n$, each with $n$ leaves, plus one star $T$ with $n$ leaves, where each edge in $S_1\cup\dots\cup S_n$ crosses each edge in $T$. Let $r$ be the root of $T$. Then $B:=\{(f,e):f\in E(S_1\cup\dots\cup S_n),e\in E(T)\}$ is a bearing of $D$, such that for each edge $vw\in E(S_1\cup\dots\cup S_n)$, the set $\{r\}$ is a $B$-cover of $vw$ in $G-v-w$ of size 1, and for each  edge $rv\in E(T)$, $\emptyset$ is a $B$-cover of $rv$ in $G-r-v$ of size 0. So $D$ is 1-gap-cover-planar. On the other hand, $D$ is not $k$-gap-planar for $n\gg k$, since every edge in $S_1$ crosses every edge in $T$, and $D$ is not $k$-cover-planar for $n\gg k$, since each edge in $T$ is crossed by a matching of $n$ edges.

\begin{figure}
    \centering
    \includegraphics{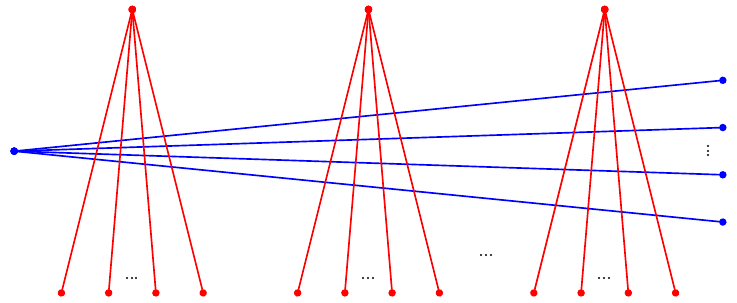}
    \caption{A 1-gap-cover-planar drawing that is neither $O(1)$-gap-planar nor $O(1)$-cover planar.}
    \label{GapCoverPlanar}
\end{figure}

The next lemma is the heart of the paper. 

\begin{lem}
\label{ShallowMinorGapCover}
Every $r$-shallow minor of a $k$-gap-cover-planar graph $G$ is $(2r+1)k$-gap-cover-planar.
\end{lem}

\begin{proof}
Fix a $k$-gap-cover-planar drawing $D_0$ of $G$ and a bearing $B_0$ of $D_0$, such that for each edge $vw\in E(G)$, there is a $B_0$-cover $A_{vw}$ of $vw$ in $G-v-w$ of size at most $k$.

Let $\mu$ be an $r$-shallow model of a graph $H$ in $G$. 
For each vertex $v\in V(H)$, $\mu(v)$ is a connected subgraph of $G$ with radius at most $r$; let $v_0$ be a centre of $\mu(v)$. 
For each edge $e\in E(G)$, 
let $A'_e$ be the set of vertices $v\in V(H)$ such that $A_e \cap V(\mu(v))\neq\emptyset$. Since each vertex in $A_e$ is in at most one subgraph $\mu(v)$,
we have $|A'_e|\leq |A_e|\leq k$. 

We now construct a drawing $D$ of $H$. 
Embed each vertex $v$ of $H$ at $v_0$. 
Let $\epsilon$ be a positive real number. 
For each edge $e\in E(G)$, let $C_e:=\{x\in\mathbb{R}^2:\dist(x,e)\leq\epsilon\}$.
There exists $\epsilon>0$ such that 
for each edge $e\in E(G)$, the only edges that intersect $C_e$ cross $e$ or are incident to an endpoint of $e$.

Embed each edge $vw\in E(H)$ as follows. 
Since $\mu$ is $r$-shallow, there is a path $(v_0,v_1,\dots,v_a)$ in $\mu(v)$, 
and there is a path $(w_0,w_1,\dots,w_b)$ in $\mu(w)$, such that $v_aw_b\in E(G)$ and $a\leq r$ and $b\leq r$.
Let $P_{vw}$ be the walk 
$(v_0,v_1,\dots,v_{a-1},v_a,w_b,w_{b-1},\dots,w_1,w_0)$. 
Embed $vw$ as a curve from $v_0$ to $w_0$ in 
$$C_{v_0v_1}\cup C_{v_1v_2}\cup\dots\cup C_{v_{a-1}v_a}\cup C_{v_aw_b}\cup C_{w_bw_{b-1}}\cup\dots\cup C_{w_1w_0}.$$
We obtain a drawing $D$ of $H$. 
Let $$A''_{vw}:= 
A'_{v_0v_1}\cup A'_{v_1v_2}\cup\dots\cup A'_{v_{a-1}v_a}\cup A'_{v_aw_b}\cup A'_{w_bw_{b-1}}\cup\dots\cup A_{w_1w_0}.$$

We now construct a bearing $B$ of $D$. Consider edges $e,f\in E(H)$ with $\{e,f\}\in D^\times$. For all edges $e_0,f_0\in E(G)$ such that $\{e_0,f_0\}\in D_0^\times$ and $e_0\in E(P_{e})$ and $f_0\in E(P_f)$, if $(e_0,f_0)\in B_0$ then add $(e,f)$ to $B$, and if $(f_0,e_0)\in B_0$ then add $(f,e)$ to $B$. Since $\{e,f\}\in D^\times$, by construction, 
$(e_0,f_0)$ or $(f_0,e_0)$ (possibly both) is in $B_0$. So $B$ is a bearing of $D$.

Consider an edge $e=vw$ of $H$. 
We claim that $A''_e$ is a cover of $e$. 
Consider an edge $e'=v'w'$ of $H$ with $(e,e')\in D^\times$.
By the choice of $\epsilon$, 
(a) $P_e$ and $P_{e'}$ have a vertex $x$ in common, or
(b) there is an edge $xy$ in $P_e$ and an edge $x'y'$ in $P_{e'}$ with $(xy,x'y')\in B$.
In case (a), without loss of generality, $x\in \mu(v)$, implying $v$ is an endpoint of both $e$ and $e'$, which contradicts that $\{e,e'\}\in D^\times$. 
In case (b), without loss of generality, $x'\in A_{e}$. 
By construction, $x\in \mu(v')\cup\mu(w')$ implying $\{v',w'\}\cap A''_e\neq\emptyset$.
So $A''_e$ is a cover for $e$. 
By construction, $|A''_e|\leq (2r+1)k$. 

So $D$ is a $(2r+1)k$-gap-cover-planar drawing of $H$, and
$H$ is $(2r+1)k$-gap-cover-planar.
\end{proof}

Now consider the extremal question for $k$-cover-planar graphs and $k$-gap-cover graphs. Pinchasi (see \citep[Lemma~4.1]{AFPS14}) proved that $k$-cover-planar graphs have density at most 
$d_{k} :=  \frac{3(k+1)^{k+1}}{k^k}$. The following variation of this proof establishes the same bound for $k$-gap-cover-planar graphs.

\begin{lem}
\label{ExtremalGapCover}
Every $k$-gap-cover-planar graph $G$ has density at most $d_k$. 
\end{lem}

\begin{proof}
Let $n:=|V(G)|$ and $m:=|E(G)|$. Let $D$ be a $k$-gap-cover-planar drawing of $G$. Let $B$ be a bearing of $D$, such that for each edge $vw\in E(G)$, $A_{vw}$ is a $B$-cover of $vw$ in $G-v-w$. Choose each vertex of $G$ independently with probability $p:=\frac{1}{k+1}$. Let $H$ be the subgraph of $G$ where $V(H)$ is the set of chosen vertices, and $E(H)$ is the set of edges $uv$ in $G$ such that $u$ and $v$ are chosen, but no vertex in $A_{uv}$ is chosen. Let $n^{*}$ and $m^{*}$ be the expected number of vertices and edges in $H$, respectively. By definition, $n^{*} = pn$. For each edge $uv\in E(G)$, the probability that $uv\in E(H)$ equals $p^2(1-p)^{|A_{uv}|}\geq p^2(1-p)^k$. Thus $m^{*} \geqslant p^{2}(1 - p)^{k}m$. Suppose for the sake of contradiction there exists $\{e,f\}\in D^\times$ with $e,f\in E(H)$. Without loss of generality, $(e,f)\in B$. By assumption, some endpoint $x$ of $f$ is in $A_e$. Since $e\in E(H)$, we have $A_e\cap V(H)=\emptyset$, implying $x\not\in V(H)$ and $f\not\in E(H)$, which is a contradiction. Thus, two edges in $H$ may cross only if they are incident to a common vertex. By the Hanani--Tutte Theorem (see \citep{Tutte70}), $H$ is planar. Hence, $|E(H)|\leq 3|V(H)|$, implying $p^{2}(1 - p)^{k}m \leqslant m^{*} \leqslant 3n^{*} = 3pn$ and $m \leqslant \frac{3}{p(1 - p)^{k}}n =d_{k}n$, as desired. 
\end{proof}

Note that $d_k < 3e(k+1)< 9(k+1)$. \cref{ShallowMinorGapCover,ExtremalGapCover} thus imply:

\begin{cor}
\label{GapCoverPlanarExpansion}
For every $k$-gap-cover-planar graph $G$, 
$$\nabla_r(G) \leq d_{(2r+1)k} < 18(r+1)(k+1).$$
\end{cor}

This result says that for each $k\geq 0$, the class of $k$-gap-cover-planar graphs has linear expansion. Since every $k$-gap-planar graph is $k$-gap-cover-planar, \cref{GapCoverPlanarExpansion} implies and strengthens \cref{LinearExpansion}. It is curious that the only known proof of \cref{LinearExpansion} goes via the more general class of $k$-gap-cover-planar graphs. It is instructive to consider what goes wrong with a direct proof along the lines of \cref{ShallowMinorGapCover}. Consider a $k$-gap-planar drawing of a graph $G$. Consider an edge $xy$ of $G$, where $x$ is the central vertex of a branch set $T_v$ for some $v\in V(H)$. It is possible that many edges of $G$ cross $xy$, and $xy$ is in the path from $x$ to many branch sets $T_w$ with $vw\in E(H)$. Using the  strategy for drawing $H$ in the proof of \cref{ShallowMinorGapCover}, all the edges that previously crossed $xy$ will cross all the edges $vw$ of $H$ that use $xy$. The resulting crossing graph has a large complete bipartite graph, and is therefore not $k'$-gap-planar with bounded $k'$. This example illustrates the usefulness of considering $k$-gap-cover-planar graphs. 

\section{Shallow Topological Minors}
\label{TopoExpansionSection}

We now prove our bound on $\widetilde{\nabla}_r(G)$ for $k$-gap-planar graphs (\cref{LinearTopoExpansion}). The proof is a slight generalisation of an argument by \citet{NOW11}, who showed that $k$-planar graphs $G$ satisfy $\widetilde{\nabla}_r(G)\leq O(\sqrt{kr})$. 

\begin{lem}
\label{SubdivTrick}
If $G$ is $k$-gap-planar and a subgraph $G'$ of $G$ is a $(\leq c)$-subdivision of some graph $H$, then $H$ is $(c+1)k$-gap-planar. 
\end{lem}

\begin{proof}
Consider a $k$-gap-planar drawing of $G$. Delete vertices or edges of $G$ not in $G'$. Consider each edge $vw\in E(H)$, which corresponds to a path $v=x_0,x_1,\dots,x_d,x_{d+1}=w$ in $G$, where $d\leq c$. Each of $x_1,\dots,x_d$ is used by no other edge of $H$. Draw $vw$ by following the path $v,x_1,\dots,x_d,w$. We obtain a drawing of $H$. Using the above notation, for each crossing in $G$ charged to $x_ix_{i+1}$, charge this crossing to $vw$. So $vw$ is charged for at most $(d+1)k\leq (c+1)k$ crossings. Hence $H$ is $(c+1)k$-gap-planar. 
\end{proof}


\begin{proof}[Proof of \cref{LinearTopoExpansion}]
Say a graph $H$ is an $r$-shallow topological minor of a $k$-gap-planar graph $G$. That is, a $(\leq 2r)$-subdivision $G'$ of $H$ is a subgraph of $G$. By \cref{SubdivTrick}, $H$ is $(2r+1)k$-gap-planar.  
By \cref{Extremal}, $H$ has density at most $8\sqrt{(2r+1)(k+1)}$. 
Thus $\widetilde{\nabla}_r(G)\leq 8\sqrt{(2r+1)(k+1)}$.
\end{proof}

Since every $r$-shallow topological minor is an  $r$-shallow minor, 
\cref{ShallowMinorGapCover} implies:

\begin{cor}
\label{GapCoverTopoMinor}
Every $r$-shallow topological minor of a $k$-gap-cover-planar graph is $(2r+1)k$-gap-cover-planar.
\end{cor}

\cref{ExtremalGapCover,GapCoverTopoMinor} imply:

\begin{cor}
\label{GapCoverTopoExpansion}
For every $k$-gap-cover-planar graph $G$, $$\widetilde{\nabla}_r(G)\leq d_{(2r+1)k} < 18(r+1)(k+1).$$
\end{cor}






\section{Surface Extensions}
\label{Surfaces}

The above results generalise for drawings on arbitrary surfaces. A surface $\Sigma$ obtained from a sphere by adding $c$ cross-caps and $h$ handles has \defn{Euler genus} $c+2h$. The definition of `drawing in the plane' generalises for `drawing in a surface $\Sigma$'. The \defn{Euler-genus $\eg(G)$} of a graph $G$ is the minimum Euler genus of a surface in which $G$ has a drawing with no crossings (see \citep{MoharThom}). The definitions of `$k$-planar drawing', `$k$-gap-planar drawing', `$k$-cover-planar drawing' and `$k$-gap-cover-planar drawing' generalise for any surface. 
A graph $G$ is \defn{$(g,k)$-planar} if $G$ has a $k$-planar drawing in a surface of Euler genus at most $g$ (see \citep{DEW17}).
A graph $G$ is \defn{$(g,k)$-gap-planar} if $G$ has a $k$-gap-planar drawing in a surface of Euler genus at most $g$. 
A graph $G$ is \defn{$(g,k)$-cover-planar} if $G$ has a $k$-cover-planar drawing in a surface of Euler genus at most $g$.
A graph $G$ is \defn{$(g,k)$-gap-cover-planar} if $G$ has a $k$-gap-cover-planar drawing in a surface of Euler genus at most $g$. The $g=0$ case corresponds to drawings in the plane. 

For an integer $g\geq 0$, let 
$c_g := \max\{ 3, \tfrac14 ( 5 + \sqrt{24g+1})\} \leq \sqrt{2g+9}$.  
\citet[Lemma~4.6]{OOW19} showed that if an $n$-vertex $m$-edge graph $G$ has a drawing in a surface of Euler genus $g$ with at most $km$ crossings, then $m\leq \sqrt{8k+4}\,c_g n$. The proof is analogous to \cref{Extremal} using an extension of the Crossing Lemma for surfaces. This result is applicable for $(g,k)$-gap-planar graphs. The next lemma follows.

\begin{lem}
\label{ExtremalSurface}
Every $(g,k)$-gap-planar graph has density at most $\sqrt{(8k+4)(2g+9)}$. 
\end{lem}

Using the fact that every $n$-vertex graph with Euler genus $g$ has treewidth at most $2\sqrt{(2g+3)n}$ \citep{DMW17}, the proof of \cref{Treewidth} generalises to show the following:

\begin{lem}
\label{gkGapTreewidth}
For every $(g,k)$-gap-planar graph $G$ on $n$ vertices, 
$$\tw(G)\leq 4(2g+3)^{1/2} (2k+1)^{3/4}n^{1/2}.$$
\end{lem}


The next lemma is proved in exactly the same way as \cref{SubdivTrick}.


\begin{lem}
\label{SubdivTrickSurface}
If a graph $G$ is $(g,k)$-gap-planar and a subgraph of $G$ is a $(\leq c)$-subdivision of some graph $H$, then $H$ is $(g,(c+1)k)$-gap-planar. 
\end{lem}

\begin{thm}
\label{LinearTopoExpansionSurface}
For every $(g,k)$-gap-planar graph $G$, 
 $$\widetilde{\nabla}_r(G)\leq \sqrt{(2r+1)(8k+4)(2g+9)} .$$
\end{thm}

\begin{proof}
Say a graph $H$ is an $r$-shallow topological minor of $G$. That is, a $(\leq 2r)$-subdivision $G'$ of $H$ is a subgraph of $G$. By \cref{SubdivTrickSurface}, $H$ is $(g,(2r+1)k)$-gap-planar.  By \cref{ExtremalSurface}, $H$ has density at most 
$\sqrt{(8(2r+1)k+4)(2g+9)}
\leq \sqrt{(2r+1)(8k+4)(2g+9)}$. The result follows.
\end{proof}

The next lemma is proved in exactly the same way as \cref{ShallowMinorGapCover}.

\begin{lem}
\label{ShallowMinorGapCoverSurface}
Every $r$-shallow minor of a $(g,k)$-gap-cover-planar graph is 
$(g,(2r+1)k)$-gap-cover-planar.
\end{lem}

Recall that the proof of \cref{ExtremalGapCover} depends on the Hanani-Tutte Theorem, which does not generalise for all surfaces~\citep{FK19}. However, the proof of \cref{ExtremalGapCover} does not require the full power of the Hanani-Tutte Theorem. It only requires that if a  graph $G$ has a drawing in the plane such that no independent edges cross, then $G$ has density at most 3. We now show this fact generalises for any surface.

For a graph $G$, let \defn{$\eg^*(G)$} be the minimum Euler genus of a surface in which $G$ has a drawing such that no independent edges cross (that is, if two edges cross then they are incident to a common vertex). 

\begin{lem}
\label{MinorClosed}
For any integer $g\geq 0$, the class of graphs $G$ with $\eg^*(G)\leq g$ is minor-closed.
\end{lem}

\begin{proof}
    Say $\eg^*(G)\leq g$. So $G$ has a drawing in a surface $\Sigma$ with Euler genus at most $g$ such that no independent edges cross. Edge deletions and vertex deletions maintain this property. Let \defn{$G/vw$} be the graph obtained from $G$ by contracting edge $vw\in E(G)$. We now show that $\eg^*(G/vw)\leq g$. Let $D$ be a disc in $\Sigma$ containing $vw$, such that the only edges that intersect $D$ are incident to $v$ or $w$, or cross $vw$. Consider an edge $xy\in E(G)$ that crosses $vw$. Without loss of generality, $y\in\{v,w\}$ and $x\not\in\{v,w\}$. Starting at $x$, let $p$ be the first point on $xy$ that is also in $D$. Delete the portion of $xy$ between $p$ and $y$ (keeping points in the intersection of $xy$ and other edges). Now contract $vw$ into a new vertex $z$, and embed $z$ in the interior of $D$. For each edge $wx$ (or $vx$) in $G$, embed edge $zx$ within $D$ plus the portion of $wx$ (or $vx$) outside of $D$ in the original embedding, as illustrated in \cref{Contracting}. If two edges $e_1$ and $e_2$ cross in this drawing of $G/vw$, then $e_1$ and $e_2$ are incident to $z$, or the edges in $G$ corresponding to $e_1$ and $e_2$ cross in the drawing of $G$, implying that $e_1$ and $e_2$ are incident to a common vertex. Thus we obtain a drawing of $G/vw$ in $\Sigma$ such that no independent edges cross. So the class of graphs $G$ with $\eg^*(G)\leq g$ is minor-closed. 
\end{proof}

\begin{figure}[!h]
\includegraphics{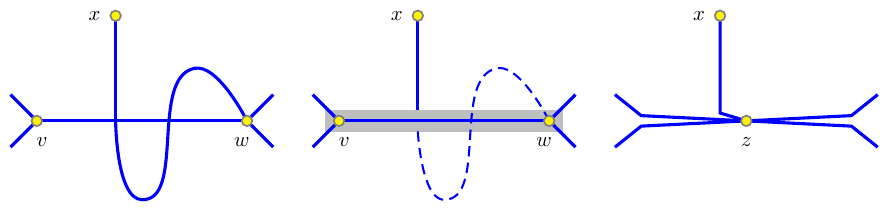}
\caption{Contracting the edge $vw$ in the proof of \cref{MinorClosed}.}
\label{Contracting}
\end{figure}

For a graph $G$, \citet{FK22} defined \defn{$\eg_0(G)$} to be the minimum Euler genus of a surface in which $G$ has a drawing such that any two independent edges cross an even number of times. By definition, every graph $H$ satisfies $\eg_0(H)\leq \eg^*(H) \leq \eg(H)$. \citet{FK22} showed that $\eg_0(K_{3,t}) = \eg(K_{3,t}) = \ceil{\frac{t-2}{2}}$. Thus 
\begin{equation}
    \label{K3t}
    \eg^*(K_{3,t})=\ceil*{\tfrac{t-2}{2}}.
\end{equation}

\begin{lem}
\label{DensityBlahEulerGenus}
Every graph $G$ has density at most $\eg^*(G)+ 3152$.
\end{lem}

\begin{proof}
\citet{KP10} showed that for $t\geq 6300$ every $K_{3,t}$-minor-free graph has density at most $\frac12(t+3)$. Let $g:=\eg^*(G)$ and $t:=\max\{2g+3,6300\}$. By \eqref{K3t}, $\eg^*(K_{3,t}) \geq \eg(K_{3,2g+3})=\ceil{\frac{2g+3-2}{2}}=g+1$. By \cref{MinorClosed}, $G$ is $K_{3,t}$-minor-free. Hence, $G$ has density at most $\frac12(t+3)\leq g+3152$.
\end{proof}

The proof of \cref{ExtremalGapCover} generalises to show the following result, where the graph $H$ in the proof satisfies $\eg^*(H)\leq g$ instead of being planar, and we use \cref{DensityBlahEulerGenus} instead of the Hanani--Tutte Theorem.

\begin{lem}
\label{ExtremalGapCoverSurface}
Every $(g,k)$-gap-cover-planar graph has density at most $3(g+3152)(k+1)$. 
\end{lem}

\cref{ShallowMinorGapCoverSurface,ExtremalGapCoverSurface} imply:

\begin{thm}
\label{LinearExpansionSurface}
For every $(g,k)$-gap-cover-planar graph $G$, 
 $$\nabla_r(G)\leq   6(g+3152)(k+1)(r+1).$$
\end{thm}

Since every $r$-shallow topological minor is an  $r$-shallow minor, 
\cref{ShallowMinorGapCoverSurface} implies:

\begin{cor}
\label{GapCoverTopoMinorSurface}
Every $r$-shallow topological minor of a $(g,k)$-gap-cover-planar graph is $(g,(2r+1)k)$-gap-cover-planar.
\end{cor}

\cref{ExtremalGapCoverSurface,GapCoverTopoMinorSurface} imply:

\begin{cor}
\label{GapCoverTopoExpansionSurface}
For every $(g,k)$-gap-cover-planar graph $G$, 
$$\widetilde{\nabla}_r(G) < 6(g+3152)(r+1)(k+1).$$
\end{cor}

\section{Colouring}
\label{ColouringNumbers}

We now explore applications of our results to graph colouring. First note that \cref{Extremal} implies that every $k$-gap-planar graph $G$ is $16\sqrt{k+1}$-degenerate, and thus $\chi(G)\leq 16\sqrt{k+1}+1$. This bound is within a constant factor of best possible, since any straight-line drawing of $K_n$ in the plane is $\binom{n}{2}$-planar and thus $\binom{n}{2}$-gap-planar. More generally, by \cref{ExtremalSurface}, every $(g,k)$-gap-planar graph $G$ is $2\sqrt{(8k+4)(2g+9)}$-degenerate, and thus $\chi(G)\leq  2\sqrt{(8k+4)(2g+9)}+1$. Even more generally, 
by \cref{ExtremalGapCoverSurface}, every $(g,k)$-gap-cover-planar graph $G$ is $6(g+3152)(k+1)$-degenerate, and thus 
$\chi(G)\leq 6(g+3152)(k+1)+1$.

We now focus on more general types of graph colouring. \citet{KY03} introduced the following definition. For a graph $G$, total order $\preceq$ of $V(G)$, and integer $r\geq 0$, a vertex $w\in V(G)$ is \defn{$r$-reachable} from a vertex $v\in V(G)$ if there is a $vw$-path $P$ in $G$ of length at most $r$, such that $w\preceq v\prec x$ for every internal vertex $x$ in $P$. Let \defn{$\sreach_r(G,\preceq,v)$} be the set of $r$-reachable vertices from $v$. 
For a graph $G$ and integer $r\geq 0$, the (\defn{strong}) \defn{$r$-colouring number $\scol_r(G)$} is the minimum integer such that there is a total order~$\preceq$ of $V(G)$ with $|\sreach_r(G,\preceq,v)|\leq \scol_r(G)$ for every vertex $v$ of $G$. 



Generalised colouring numbers are important because they characterise bounded expansion classes \citep{Zhu09}, they characterise nowhere dense classes \citep{GKRSS18}, and have several algorithmic applications such as the constant-factor approximation algorithm for domination number by \citet{Dvorak13}, and the almost linear-time model-checking algorithm of \citet{GKS17}. Generalised colouring numbers also provide upper bounds on several graph parameters of interest. For example, a proper vertex-colouring of a graph $G$ is \defn{acyclic} if the union of any two colour classes induces a forest; that is, every cycle is assigned at least three colours. The \defn{acyclic chromatic number} $\chi_\text{a}(G)$ of a graph $G$ is the minimum integer $k$ such that $G$ has an acyclic $k$-colouring. Acyclic colourings are qualitatively different from colourings, since every  graph with bounded acyclic chromatic number has bounded average degree. \citet{KY03} proved that every graph $G$ satisfies
\begin{equation}
    \label{ACN}
    \chi_\text{a}(G)\leq \scol_2(G).
\end{equation}
Other examples include game chromatic number \citep{KT94,KY03}, Ramsey numbers \citep{CS93}, oriented chromatic number \citep{KSZ-JGT97}, arrangeability~\citep{CS93}, etc. 


What can be said about $\scol_r(G)$ where $G$ is in one of the graph classes considered in this paper? First note that \citet{HOQRS17} proved that $\scol_r(G) \leq 5r+1$ for every  planar graph $G$, and $\scol_r(G) \leq (4g+5)r + 2g+1$ for every graph $G$ with Euler genus $g$. More generally, \citet{vdHW18} showed that $\scol_r(G)\leq (2r+1)\ltw(G)$ for any graph $G$. \citet{HKW} showed that $k$-cover-planar graphs $G$ have $\ltw(G)\leq f(k)$, implying $\scol_r(G)\leq (2r+1) f(k)$. But the function $f$ here is exponential in $k$, and this method fails even for 1-gap-planar graphs, which have unbounded layered treewidth~\citep{HIMW24}. We obtain reasonable bounds on $\scol_r(G)$ for all the classes of graphs considered in this paper by exploiting a connection with expansion.









\citet{Zhu09} first showed that $\scol_r(G)$ is upper bounded by a function of $\nabla_r(G)$ (or more precisely, by a function of 
$\nabla_{(r-1)/2}(G)$). The best known bounds follow from results of \citet{GKRSS18} (see the survey by \citet{Siebertz25}). In particular, for every integer $r\geq 1$,
\begin{equation}
    \label{scol-nabla}
    \scol_r(G) \leq  (6r)^r \widetilde{\nabla}_{r-1}(G)^{3r}.
\end{equation}

By \cref{LinearTopoExpansion}, for every $k$-gap-planar graph $G$, 
 $$\widetilde{\nabla}_r(G)\leq 8\sqrt{(2r+1)(k+1)}.$$
By \eqref{scol-nabla}, for $r\geq 1$, 
$$\scol_r(G) \leq 
(6r)^r ( 8(2r-1)^{1/2}(k+1)^{1/2} )^{3r}
\leq (8689\, r^{5/2})^r (k+1)^{3r/2}
.$$
In particular,  with $r=2$, 
$$\chi_\text{a}(G)\leq \scol_2(G) \leq O((k+1)^3).$$
More generally, by 
\cref{LinearTopoExpansionSurface}, for every $(g,k)$-gap-planar graph $G$, 
 $$\widetilde{\nabla}_r(G)\leq \sqrt{(2r+1)(8k+4)(2g+9)} .$$
By \eqref{scol-nabla} for $r\geq 1$, 
\begin{align*}
    \scol_r(G) & \leq  (6r)^r ((2r-1)(8k+4)(2g+9))^{3r/2}\\
& \leq (17 r^{5/2} )^{r} (8k+4)^{3r/2}(2g+9)^{3r/2}.
\end{align*}
In particular, with $r=2$, 
$$\chi_\text{a}(G)\leq \scol_2(G) \leq O((k+1)^3(g+1)^3).$$

Finally, consider $(g,k)$-gap-cover-planar graphs $G$, which by \cref{GapCoverTopoExpansionSurface}, satisfy
 $$\widetilde{\nabla}_r(G)\leq 6(g+3152)(r+1)(k+1).$$
By  \eqref{scol-nabla},   for $r\geq 1$, 
 \begin{align*}
    \scol_r(G) 
&    \leq    (6r)^r ( 6(g+3152)(r+1)(k+1) )^{3r}\\
 &   <    (6^4(r+1)^{5/2})^r (g+3152)^{3r}(k+1)^{3r}.    
\end{align*}
In particular, with $r=2$, 
$$\chi_\text{a}(G)\leq \scol_2(G) \leq O((g+1)^6(k+1)^6).$$

\section{Open Problems}

The following open problems arise from this work:

\begin{itemize}
\item Does $\eg^*(H)=\eg(H)$ for every graph $H$? See \citep{FK22,SSSV-Algo96,Schaefer22} for related results.

\item Do $k$-gap-cover-planar graphs on $n$ vertices have treewidth $O_k(\sqrt{n})$? Here is the best known upper bound.

\begin{prop}
    Every $k$-gap-cover-planar graph on $n$ vertics has treewidth $O( k^{3/4} \sqrt{n} \polylog(n))$.
\end{prop}

\begin{proof}
\citet{CC15} showed that any graph $G$ with treewidth $t$ has a topological minor $H$ with maximum degree at most 3 such that $\tw(H)\geq \Omega( t/ \polylog(t))$. Let $G_0$ be the subgraph of $G$ that is a subdivision of $H$. Thus $G_0$ has maximum degree at most 3. Since treewidth is invariant under subdivision, $\tw(G_0)=\tw(H)\geq \Omega( t/ \polylog(t))$. 

Apply this result to a $k$-gap-cover-planar graph $G$ with $n$ vertices and treewidth $t$. So $G$ has a subgraph $G_0$ with maximum degree at most 3, and $\tw(G_0)\geq ct/ \polylog(t)$, for some constant $c>0$. Observe that every subgraph of a $k$-gap-cover-planar graph is $k$-gap-cover-planar, and every $k$-gap-cover-planar graph with maximum degree at most $\Delta$ is $k\Delta$-gap-planar. So $G_0$ is $k$-gap-cover-planar and $3k$-gap-planar. By \cref{Treewidth}, 
$$ct/\polylog(t) \leq \tw(G_0)\leq 21(3k+1)^{3/4}n^{1/2},$$
implying 
$$\tw(G)=t \leq \tfrac{21}{c}(3k+1)^{3/4}n^{1/2} \polylog(t)
\leq \tfrac{21}{c}(3k+1)^{3/4}n^{1/2} \polylog(n),$$
as desired.
\end{proof}

\item This paper shows that $k$-gap-planar graphs (and various extensions) have linear expansion. Still, it is desirable to have a more structural description of $k$-gap-planar graphs (in the spirit of the product structure theorem for $k$-planar graphs~\citep{DMW23,HW24,DHSW24} or for $k$-cover-planar graphs~\citep{HKW}). It is not clear what such a structural description should be. As a guide, one would expect that such a result would lead to positive solutions to the following open problems: Do $k$-gap-planar graphs have queue-number at most some function $f(k)$? Do $k$-gap-planar graphs have non-repetitive chromatic number at most some function $f(k)$? See \citep{HW24,DMW23} for the corresponding results for $k$-planar graphs. 
\end{itemize}


\subsection*{Acknowledgements} Thanks to André Kündgen for helpful discussions about colouring $k$-gap-planar graphs. Thanks to Louis Esperet for helpful discussions about \cref{ER}.
Thanks to Marcus Schaefer for helpful discussions about the first open problem above. 

{\fontsize{10pt}{11pt}\selectfont
\bibliographystyle{DavidNatbibStyle}
\bibliography{DavidBibliography,temp}}

\end{document}